\documentclass[12pt]{article}

\usepackage{amsmath, amssymb}
\usepackage{amsthm}

\usepackage{color}
\definecolor{darkred}{RGB}{139,0,0}
\definecolor{darkgreen}{RGB}{0,100,0}
\definecolor{darkmagenta}{RGB}{139,0,139}
\definecolor{darkpurple}{RGB}{110,0,180}
\definecolor{darkblue}{RGB}{40,0,200}
\definecolor{darkorange}{RGB}{255,140,0}

\newcommand{\bsx}{\boldsymbol{x}}
\newcommand{\bsh}{\boldsymbol{h}}
\newcommand{\bsa}{\boldsymbol{a}}
\newcommand{\bsb}{\boldsymbol{b}}
\newcommand{\bszero}{\boldsymbol{0}}

\newcommand{\bsg}{\boldsymbol{g}}
\newcommand{\bsv}{\boldsymbol{v}}

\newcommand{\bsy}{\boldsymbol{y}}

\newcommand{\bsw}{\boldsymbol{w}}
\newcommand{\rd}{\,{\rm d}}

\newcommand{\NN}{\mathbb{N}}
\newcommand{\ZZ}{\mathbb{Z}}
\newcommand{\CC}{\mathbb{C}}
\newcommand{\cP}{\mathcal{P}}

\newcommand{\cH}{\mathcal{H}}
\newcommand{\icomp}{\mathtt{i}}

\newtheorem{theorem}{Theorem}
\newtheorem{proposition}[theorem]{Proposition}
\newtheorem{remark}[theorem]{Remark}
\newtheorem{lemma}[theorem]{Lemma}

\begin{document}

\title{A note on  Korobov lattice rules for integration of analytic functions}

\author{Friedrich Pillichshammer\thanks{F. Pillichshammer is supported by the Austrian Science Fund (FWF) Project F5509-N26, which is a part of the Special Research Program ``Quasi-Monte Carlo Methods: Theory and Applications''.}
}

\date{}

\maketitle

\begin{abstract}
We study numerical integration for a weighted Korobov space of analytic periodic functions for which the Fourier coefficients decay exponentially fast. In particular, we are interested in how the error depends on the dimension $d$. Many recent papers deal with this problem or similar problems and provide matching necessary and  sufficient conditions for various notions of tractability. In most cases even simple algorithms are known which allow to achieve these notions of tractability. However, there is a gap in the literature: while for the notion of exponential-weak tractability one knows matching necessary and sufficient conditions, so far no explicit algorithm has been known which yields the desired result. 

In this paper we close this gap and prove that Korobov lattice rules are suitable algorithms in order to achieve exponential-weak tractability for integration in weighted Korobov spaces of analytic periodic functions. 
\end{abstract}

\centerline{\begin{minipage}[hc]{130mm}{
}
\end{minipage}}

\section{Introduction}

Many recent papers study numerical integration and approximation for suitably weighted  Korobov spaces, cosine spaces or Hermite spaces of analytic functions (see, for example, \cite{DKPW14,DLPW11,IKP18,IKPW16a,IKPW16b,KPW14a,KPW14b,LX16,W19,XX16,X20}). It is well known that for these problems exponential error convergence rates can be achieved and, using suitably chosen weights, even uniform exponential convergence rates. For problems with very high dimension it is of utmost importance to know also the dependence of the error bounds on the dimension. This question is the subject of tractability (see \cite{NW1,NW2,NW3}). In the context of exponential convergence rates it turned out to be reasonable to study notions of so-called EXP-tractability, where the prefix ``EXP'' stands for ``exponential'' and where, compared to the standard algebraic (ALG-) tractability notions, $\log \varepsilon^{-1}$ takes the role of $\varepsilon^{-1}$ in the requested bounds on the information complexity. For most problems studied one knows ``if and only if'' conditions on the implied weights which guarantee the respective notions of tractability and in most cases one even knows simple algorithms which allow to achieve these notions of EXP-tractability. These algorithms are based on regular grids of different mesh-sizes. 

While for one of the mildest notions of EXP-tractability, namely EXP-weak tractability, one knows matching necessary and sufficient conditions on the weights, just for this case no explicit algorithm is known which leads to the desired result. The corresponding proof of sufficient conditions for achieving EXP-weak tractability requires a detour to $L_2$-approximation and heavy machinery from tractability theory in combination with thorough estimates of the eigenvalues of a suitable operator, see \cite[Section~9]{DKPW14}.

In this short note we consider integration in Korobov spaces and show that actually (Korobov) lattice rules are suitable algorithms to achieve EXP-weak tractability. As a by-product also ``constructive" proofs for the standard ALG-tractability notions can be obtained, even with an improvement of the currently known bound on the  $\varepsilon$-exponent of ALG-strong polynomial tractability. 

\section{Basic definitions and preliminaries}    

\paragraph{The function space.} As in \cite{DKPW14,KPW14b} let $$\bsa=(a_j)_{j \ge 1} \quad \mbox{ and }\quad \bsb=(b_j)_{j \ge 1}$$ be two sequences of positive reals (the weights) for which we assume that we have 
\begin{equation}\label{cond:weightab}
0<a_1 \le a_2 \le \ldots \quad \mbox{ and } \quad b_{\ast}=\inf b_j >0.
\end{equation}
We also write $a_\ast:=\inf a_j =a_1$. Now we fix $\omega \in (0,1)$ and put $$\rho_{\bsa,\bsb}(\bsh):=\omega^{\sum_{j=1}^d a_j |h_j|^{b_j}}\quad \mbox{ for $\bsh=(h_1,\ldots,h_d)\in \ZZ^d$.}$$ Define the reproducing kernel $K_{d,\bsa,\bsb}:[0,1]^d \times [0,1]^d \rightarrow \CC$ by 
\begin{equation}\label{def:kerkorgen}
K_{d,\bsa,\bsb}(\bsx,\bsy)=\sum_{\bsh \in \ZZ^d} \rho_{\bsa,\bsb}(\bsh) \exp(2 \pi \icomp \bsh \cdot (\bsx-\bsy)) \quad \mbox{for $\bsx, \bsy \in [0,1]^d$.}
\end{equation}
It is easy to see that $K_{d,\bsa,\bsb}$ is conjugate symmetric and positive semi-definite and therefore indeed a reproducing kernel.

We denote the corresponding reproducing kernel Hilbert space with kernel $K_{d,\bsa,\bsb}$ by $\cH_{d,\bsa,\bsb}$. Functions $f \in \cH_{d,\bsa,\bsb}$ are one-periodic in each coordinate and we have $$f(\bsx)=\sum_{\bsh \in \ZZ^d} \widehat{f}(\bsh) \exp(2 \pi \icomp \bsh \cdot \bsx) \quad \mbox{ for all $\bsx \in [0,1]^d$,}$$ where $\widehat{f}(\bsh)=\int_{[0,1]^d} f(\bsx) \exp(-2 \pi \icomp \bsh \cdot \bsx) \rd \bsx$ is the $\bsh^{{\rm th}}$ Fourier coefficient of $f$, and the norm of $f \in  \cH_{d,\bsa,\bsb}$ is $$\|f\|_{d,\bsa,\bsb} =\left(\sum_{\bsh \in \ZZ^d} \omega^{- \sum_{j=1}^d a_j |h_j|^{b_j}} |\widehat{f}(\bsh)|^2 \right)^{1/2} < \infty.$$

It is well known that functions $f \in \cH_{d,\bsa,\bsb}$ are infinitely many times differentiable and even analytic (see \cite[Section~10]{DKPW14}).

\paragraph{Numerical integration.} We study numerical integration of functions from $\cH_{d,\bsa,\bsb}$. It is well known (see, e.g.,  \cite{TWW88}) that we can restrict ourselves to approximating $\int_{[0,1]^d} f(\bsx) \rd \bsx$ by means of {\it linear algorithms} $Q_{N,d}$ of the form
\[Q_{N,d}(f):=\sum_{k=0}^{N-1} w_k f(\bsx_k),\]
with coefficients $\bsw=(w_1,\ldots,w_d) \in \CC^d$ and sample points $\cP=\{\bsx_0,\bsx_1,\ldots,\bsx_{N-1}\}$  in $[0,1)^d$. We are interested in studying the {\it worst-case integration error},
$$
e(\cH_{d,\bsa,\bsb},\cP,\bsw)=\sup_{f \in \cH_{d,\bsa,\bsb} \atop
\|f\|_{d,\bsa,\bsb} \le 1} \left|\int_{[0,1]^d} f(\bsx) \rd \bsx-Q_{N,d}(f)\right|.$$ 

Let $e(N,d)$ be the {\it $N^{{\rm th}}$ minimal worst-case error},
$$
e(N,d)=\inf_{\cP, \bsw}\ e(\cH_{d,\bsa,\bsb},\cP,\bsw),
$$
where the infimum is extended over all $N$-element point sets $\cP$ in $[0,1)^d$ and over all weights $\bsw$ in $\CC^N$.
For $N=0$, the best we can do is to approximate the integral simply by zero, and the so-called {\it initial error} $e(0,d)$, which is the norm of the integral operator on $\cH_{d,\bsa,\bsb}$, equals one in the present case. Hence, the integration problem is normalized for all $d$. 

\paragraph{Exponential convergence and tractability.}

Since the integrands from $\cH_{d,\bsa,\bsb}$ are analytic, one may expect that optimal algorithms achieve exponential convergence rates for the respective worst-case errors. This is indeed the case: From \cite{DKPW14,KPW14a} we know that there exist numbers $q\in(0,1)$ and functions $p,C,M: \NN \rightarrow (0,\infty)$ such that 
\begin{equation}\label{exrate}
e(N,d) \le C(d)\,q^{(N/M(d))^{p(d)}}\ \ \ \ \mbox{for all}\ \ d,N\in\NN.
\end{equation}
We say that we have {\it exponential convergence} of $e(N,d)$ in $\cH_{d,\bsa,\bsb}$. Moreover, if the weight sequence $\bsb=(b_j)_{j \ge 1}$ tends to infinity so fast that $\sum_{j=1}^{\infty} b_j^{-1} < \infty$, then we even have {\it uniform exponential convergence} which means that we can take $p(d)=p>0$ for all $d \in \NN$ in \eqref{exrate}.

Besides the convergence rate in $N$ also the dependence of the worst-case error on the dimension $d$ is important, in particular for very high-dimensional problems. This is the subject of tractability (see \cite{NW1,NW2,NW3}). Here one studies the so-called {\it information complexity} of the problem, which is defined as $$N(\varepsilon,d)=\min\{N \in \NN \ : \ e(N,d) \le \varepsilon \, e(0,d)\} \quad \mbox{ for $\varepsilon \in (0,1)$ and $d \in \NN$.}$$

Several notions of tractability are studied which classify the growth of the information complexity in terms of $\varepsilon^{-1}$ and $d$. We call these the {\it algebraic (ALG) notions of tractability}. In the context of exponential convergence rates the notions of {\it EXP-tractability} are of particular importance. Here one is interested in the growth of $N(\varepsilon,d)$ in terms of $\log \varepsilon^{-1}$ and $d$ (see \cite{DKPW14,DLPW11,KPW14a,KPW14b}).

In the present note we only deal with EXP-weak tractability which rules out the cases for which $N(\varepsilon,d)$ depends exponentially on $d$ or $\log \varepsilon^{-1}$. We say that we have {\it EXP-weak tractability} (abbreviated {\it EXP-WT}) if $$\lim_{d+\varepsilon^{-1} \rightarrow \infty} \frac{\log N(\varepsilon,d)}{d+\log \varepsilon^{-1}} =0.$$

It is well known that the integration problem in $\cH_{d,\bsa,\bsb}$ is EXP-WT if and only if $\lim_{j \rightarrow \infty} a_j=\infty$ (see \cite[Section~9]{DKPW14}). The proof of this result is non-constructive. This is in contrast to other notions of EXP-tractability like, e.g., EXP-strong polynomial tractability, which can be achieved under certain conditions on the weights $\bsa$ and $\bsb$ by means of a simple algorithm based on regular grids. In order to achieve EXP-WT when $\lim_{j \rightarrow \infty}a_j=\infty$ no such algorithm has been known so far. 

\section{The result} 

Provided $\lim_{j \rightarrow \infty}a_j=\infty$,  we show that EC-WT can be achieved by means of {\it lattice rules} (see \cite{niesiam,SJ94}) of the form
\begin{equation}\label{def:LR}
Q_{N,d}(f)=\frac{1}{N} \sum_{k=0}^{N-1} f\left(\left\{\frac{k}{N} \bsg\right\}\right), \quad \mbox{with a suitable $\bsg \in \ZZ^d$.}
\end{equation}
Hence the underlying node set $\cP=\cP(\bsg,N)$ consists of the points $\bsx_k=\left\{\frac{k}{N} \bsg\right\}$ for $k \in G_N:=\{0,1,\ldots,N-1\}$, where the fractional-part function $\{x\}=x-\lfloor x \rfloor$ is applied component-wise, and the coefficients are $\bsw =\bsw_{\text{QMC}}=(N^{-1},\ldots,N^{-1})$. The vector $\bsg$ is usually called the {\it generating vector} of the lattice rule. Special types of lattice rules are {\it Korobov rules} where the generating vector $\bsg$ is of the form $$\bsg=\bsv_d(g):=(1,g,g^2,\ldots,g^{d-1}),\quad \mbox{with a suitable $g \in G_N$.}$$ Generating vectors of such particular form are called {\it Korobov vectors}. 

Since we always deal with coefficients $\bsw_{\text{QMC}}$ we will denote the worst-case error of a lattice rule \eqref{def:LR} simply by $e(\cH_{d,\bsa,\bsb},\cP(\bsg,N))$ from now on.
\begin{lemma}\label{le1}
The squared worst-case error of a lattice rule in $\cH_{d,\bsa,\bsb}$ is given by $$e^2(\cH_{d,\bsa,\bsb},\cP(\bsg,N)) =  \sum_{\bsh \in \ZZ^d \setminus\{\bszero\} \atop \bsh \cdot \bsg \equiv 0 \pmod{N}} \omega^{\sum_{j=1}^d a_j |h_j|^{b_j}}.$$
\end{lemma}

\begin{proof}
Using the worst-case error formula for reproducing kernel Hilbert spaces (see, for example, \cite[Theorem~3.5]{DKS13} or  \cite[Proposition~2.11]{DP10}) and the fact that $\int_{[0,1]^d} K_{d,\bsa,\bsb}(\bsx,\bsy) \rd \bsy=1$ we have
$$e^2(\cH_{d,\bsa,\bsb},\cP(\bsg,N)) = -1+\frac{1}{N^2}\sum_{k,l=0}^{N-1} K_{d,\bsa,\bsb}(\bsx_k,\bsx_l).$$
Inserting the definition of the kernel \eqref{def:kerkorgen} and interchanging the order of summation we obtain 
\begin{eqnarray*}
e^2(\cH_{d,\bsa,\bsb},\cP(\bsg,N)) & = & -1+ \sum_{\bsh \in \ZZ^d} \omega^{\sum_{j=1}^d a_j |h_j|^{b_j}}  \left(\frac{1}{N^2} \sum_{k,l=0}^{N-1} \exp(2 \pi \icomp (k-l) (\bsh \cdot \bsg)/N)) \right) \\
& = & \sum_{\bsh \in \ZZ^d\setminus\{\bszero\}} \omega^{\sum_{j=1}^d a_j |h_j|^{b_j}}  \left|\frac{1}{N} \sum_{k=0}^{N-1} \exp(2 \pi \icomp k (\bsh \cdot \bsg)/N)) \right|^2.
\end{eqnarray*}
Now the result follows because the inner exponential sum equals one if and only if $\bsh \cdot \bsg \equiv 0 \pmod{N}$ and zero otherwise.
\end{proof}

The following proposition is the key result in order to achieve EXP-WT by means of lattice rules.

\begin{proposition}\label{le2}
Let $N$ be a prime number and let $d$ be a positive integer. For $\lambda \in (0,1]$ define $A_\lambda:=\sum_{h=1}^\infty \omega^{\lambda a_\ast (h^{b_\ast}-1)} < \infty.$ Consider general generating vectors and Korobov generating vectors, respectively. The following two assertions hold:
\begin{enumerate}
\item Let $\bsg_\ast \in G_N^d$ be such that $e(\cH_{d,\bsa,\bsb},\cP(\bsg_\ast,N)) = \min_{\bsg \in G_N^d} e(\cH_{d,\bsa,\bsb},\cP(\bsg,N))$. Then for all $\lambda \in (0,1]$ we have $$e(\cH_{d,\bsa,\bsb},\cP(\bsg_\ast,N)) \le \left(\frac{1}{N} \prod_{j=1}^d (1+2 A_\lambda \omega^{\lambda a_j}) \right)^{1/(2 \lambda)}.$$
\item Let $g_\ast \in G_N$ be such that $$e(\cH_{d,\bsa,\bsb},\cP(\bsv_d(g_\ast),N)) = \min_{g \in G_N} e(\cH_{d,\bsa,\bsb},\cP(\bsv_d(g),N)).$$ Then for all $\lambda \in (0,1]$ we have $$e(\cH_{d,\bsa,\bsb},\cP(\bsv_d(g_\ast),N)) \le \left(\frac{d-1}{N} \prod_{j=1}^d (1+2 A_\lambda \omega^{\lambda a_j}) \right)^{1/(2 \lambda)}.$$
\end{enumerate}
\end{proposition}

\begin{proof}
For every $j \in \NN$ we have
\begin{eqnarray*}
\sum_{h =-\infty}^{\infty} \omega^{\lambda a_j |h|^{b_j}} = 1+2 \sum_{h=1}^\infty \omega^{\lambda a_j |h|^{b_j}}
= 1+2 \omega^{\lambda a_j} \sum_{h=1}^\infty \omega^{\lambda a_j (|h|^{b_j}-1)} \le 1+2 A_\lambda  \omega^{\lambda a_j}
\end{eqnarray*}
and hence 
\begin{eqnarray}\label{bd:sum}
\sum_{\bsh \in \ZZ^d \setminus\{\bszero\}} \omega^{\sum_{j=1}^d \lambda a_j |h_j|^{b_j}}  \le \prod_{j=1}^d \left(\sum_{h =-\infty}^{\infty} \omega^{\lambda a_j |h|^{b_j}}\right)\le \prod_{j=1}^d \left(1+2 A_\lambda  \omega^{\lambda a_j}\right).
\end{eqnarray}

Using Lemma~\ref{le1} and Jensen's inequality, which states that for any $\lambda \in (0,1]$ and non-negative reals $a_k$ it holds that $(\sum_k a_k)^\lambda \le \sum_k a_k^\lambda$ (see, e.g., \cite[pp. 100-101]{LP14}), for $\lambda \in (0,1]$ we have $$e^{2 \lambda}(\cH_{d,\bsa,\bsb},\cP(\bsg,N)) =\left( \sum_{\bsh \in \ZZ^d \setminus\{\bszero\} \atop \bsh \cdot \bsg \equiv 0 \pmod{N}} \omega^{\sum_{j=1}^d a_j |h_j|^{b_j}} \right)^{\lambda} \le \sum_{\bsh \in \ZZ^d \setminus\{\bszero\} \atop \bsh \cdot \bsg \equiv 0 \pmod{N}} \omega^{\sum_{j=1}^d \lambda a_j |h_j|^{b_j}}.$$

\begin{enumerate}
\item First we treat the case of general generating vectors. Averaging over $\bsg \in G_N^d$ and interchanging the order of summation yields 
$$
\frac{1}{N^d} \sum_{\bsg \in G_N^d} e^{2 \lambda}(\cH_{d,\bsa,\bsb},\cP(\bsg,N))  \le  
\sum_{\bsh \in \ZZ^d \setminus\{\bszero\}} \omega^{\sum_{j=1}^d \lambda a_j |h_j|^{b_j}} \left( \frac{1}{N^d} \sum_{\bsg \in G_N^d \atop \bsh \cdot \bsg \equiv 0 \pmod{N} } 1\right).
$$
Since $\bsh \not=\bszero$ and $N$ is a prime number we find that the congruence $\bsh \cdot \bsg \equiv 0 \pmod{N}$ has exactly $N^{d-1}$ solutions in $G_N^d$. Hence 
\begin{equation*}
\frac{1}{N^d} \sum_{\bsg \in G_N^d} e^{2 \lambda}(\cH_{d,\bsa,\bsb},\cP(\bsg,N))  \le \frac{1}{N}\sum_{\bsh \in \ZZ^d \setminus\{\bszero\}} \omega^{\sum_{j=1}^d \lambda a_j |h_j|^{b_j}} \le \frac{1}{N} \prod_{j=1}^d \left(1+2 A_\lambda  \omega^{\lambda a_j}\right),
\end{equation*}
where we applied \eqref{bd:sum}. In particular, there exists a $\overline{\bsg} \in G_N^d$ such that  $$e(\cH_{d,\bsa,\bsb},\cP(\overline{\bsg},N))  \le \left(\frac{1}{N}\prod_{j=1}^d \left(1+2 A_\lambda  \omega^{\lambda a_j}\right)\right)^{1/(2 \lambda)}.$$
Obviously $e(\cH_{d,\bsa,\bsb},\cP(\bsg_\ast,N)) \le e(\cH_{d,\bsa,\bsb},\cP(\overline{\bsg},N))$, and this proves the result.
\item In the case of Korobov rules we proceed in the same way as above. Averaging  over $g \in G_N=\{0,1,\ldots,N-1\}$ and interchanging the order of summation we obtain 
$$
\frac{1}{N} \sum_{g=0}^{N-1} e^{2 \lambda}(\cH_{d,\bsa,\bsb},\cP(\bsv_d(g),N))  \le  
\sum_{\bsh \in \ZZ^d \setminus\{\bszero\}} \omega^{\sum_{j=1}^d \lambda a_j |h_j|^{b_j}} \left(\frac{1}{N} \sum_{g =0 \atop \bsh \cdot \bsv_d(g)\equiv 0 \pmod{N}}^{N-1}  1\right).
$$
The congruence $\bsh \cdot \bsv_d(g)\equiv 0 \pmod{N}$ is of polynomial form $h_1 +h_2 g+h_3 g^2+\cdots +h_d g^{d-1}\equiv 0 \pmod{N}$ and has at most $d-1$ solutions in $G_N$ for nonzero $\bsh=(h_1,\ldots,h_d) \in \ZZ^d$. Hence, 
\begin{equation*}
\frac{1}{N} \sum_{g=0}^{N-1} e^{2 \lambda}(\cH_{d,\bsa,\bsb},\cP(\bsv_d(g),N))  \le  \frac{d-1}{N}
\prod_{j=1}^d \left(1+2 A_\lambda  \omega^{\lambda a_j}\right) .
\end{equation*}
From here the result follows in the same way as above.
\end{enumerate}
\end{proof}

\paragraph{EXP-WT by means of lattice rules.} Now assume that $\lim_{j \rightarrow \infty} a_j=\infty$.

Let $\varepsilon \in (0,1)$ and $d \in \NN$. Let in the following $c_d=1$ in case of general lattice rules and $c_d=d$ in case of Korobov rules. For $\lambda \in (0,1]$ let $N$ be the smallest prime number greater than or equal to $$\left\lceil c_d \,\varepsilon^{-2 \lambda} \prod_{j=1}^d \left(1+2 A_\lambda \omega^{\lambda a_j}\right)\right\rceil =:M_\lambda(\varepsilon,d).$$ 
Note that according to Bertrand's postulate we have $N \in [M_\lambda(\varepsilon,d),2 M_\lambda(\varepsilon,d))$.

According to Proposition~\ref{le2} there exists a lattice rule (with $\bsg_\ast \in G_N^d$) or even a Korobov rule (with $g_\ast \in G_N$) such that $e(\cH_{d,\bsa,\bsb},\cP(\bsg,N)) \le \varepsilon$, where $\bsg=\bsg_\ast$ or $\bsg=\bsv_d(g_\ast)$, respectively. Hence $$N(\varepsilon,d)\le N \le 2M \le 4 c_d \varepsilon^{-2 \lambda} \prod_{j=1}^d \left(1+2 A_\lambda \omega^{\lambda a_j}\right).$$ This means that for every $\lambda \in (0,1]$ we have 
\begin{equation}\label{bd:infcomp}
N(\varepsilon,d)\le 4 c_d \varepsilon^{-2 \lambda} \prod_{j=1}^d \left(1+2 A_\lambda \omega^{\lambda a_j}\right).
\end{equation}
Consequently,
\begin{eqnarray*}
\frac{\log N(\varepsilon,d)}{d+\log \varepsilon^{-1}} & \le & \frac{\log 4 +\log c_d+ 2 \lambda \log \varepsilon^{-1} +\sum_{j=1}^d \log (1+2 A_\lambda \omega^{\lambda a_j})}{d+\log \varepsilon^{-1}}\\
& \le &  \frac{\log 4}{d+\log \varepsilon^{-1}}+ \frac{\log d}{d+\log \varepsilon^{-1}}+ \frac{2 \lambda \log \varepsilon^{-1}}{d+\log \varepsilon^{-1}} +\frac{2 A_\lambda\sum_{j=1}^d \omega^{\lambda a_j}}{d+\log \varepsilon^{-1}}.
\end{eqnarray*}
Note that $\lim_{j \rightarrow \infty} a_j=\infty$ implies that $\lim_{j \rightarrow \infty} \omega^{\lambda a_j}=0$, and hence $$\lim_{d \rightarrow \infty}\frac{1}{d}\sum_{j=1}^d \omega^{\lambda a_j}=0.$$ This implies that $$\limsup_{d+\varepsilon^{-1} \rightarrow \infty} \frac{\log N(\varepsilon,d)}{d+\log \varepsilon^{-1}} \le 2 \lambda.$$ Since $\lambda \in (0,1]$ can be arbitrarily close to zero we obtain $$\lim_{d+\varepsilon^{-1} \rightarrow \infty} \frac{\log N(\varepsilon,d)}{d+\log \varepsilon^{-1}}=0$$ and this proves EXP-WT. \hfill$\qed$

\begin{remark}\rm
It is well known (see \cite{KPW14b}) that EXP-WT implies that the $N^{{\rm th}}$ minimal worst-case error tends to zero faster than any power of $N^{-1}$, i.e., $\lim_{N \rightarrow \infty} N^{\alpha} e(N,d)=0$ for all $\alpha>0$. This fact in conjunction with our result shows that such convergence rates can be achieved with (Korobov) lattice rules. Lattice rules can  even yield exponential convergence rates, see \cite{DLPW11}. However, whether lattice rules are strong enough to achieve also uniform exponential convergence or more demanding notions of tractability like, for example, EXP-strong polynomial tractablity, remains to be examined.
\end{remark}

\begin{remark}\rm
There is the more general notion of EXP-$(s,t)$-WT for reals $s,t>0$ (see, e.g., \cite[Sec.~5.2]{IKPW16b}), which means that $$\lim_{d+\varepsilon^{-1} \rightarrow \infty} \frac{\log N(\varepsilon,d)}{d^s+(\log \varepsilon^{-1})^t} =0.$$ In order to extend the above considerations to the more general setting we obviously have to assume that $t \ge 1$. Then lattice rules can achieve EXP-$(s,t)$-WT for the following cases:
\begin{itemize}
\item $s>1$ and $t \ge 1$ (without any further condition on the weights beyond \eqref{cond:weightab});
\item $s=1$ and $t \ge 1$ and $\lim_{j \rightarrow \infty} a_j =\infty$.
\item $s<1$ and $t \ge 1$ and $\sum_{j=1}^d \omega^{\lambda a_j} =o(d^s)$ for all $\lambda>0$. For example, this condition on $(a_j)_{j \ge 1}$ is satisfied for any $s>0$ if $\lim_{j \rightarrow \infty} a_j/\log j = \infty$.
\end{itemize}
\end{remark}

We close with some comments on the algebraic notions of tractability.

\begin{remark}\rm
From \eqref{bd:infcomp} one can also deduce ``constructive" proofs of the results on the algebraic notions ALG-strong polynomial, ALG-polynomial and ALG-weak tractability in \cite[Theorem~4.2]{KPW14b}. For example, the estimate \eqref{bd:infcomp} in case of lattice rules implies 
\begin{equation}\label{bd:Nstd}
N(\varepsilon,d)\le 4  \varepsilon^{-2 \lambda} \exp\left(2 A_\lambda \sum_{j=1}^d \omega^{\lambda a_j}\right) \quad \forall \lambda \in (0,1].
\end{equation}
 It is easy to see (or check \cite[Proof of Theorem~5.2]{KPW14b}) that $\sum_{j=1}^\infty \omega^{\lambda a_j} < \infty$ if and only if $A:=\lim_{j \rightarrow \infty} \frac{a_j}{\log j}$ satisfies $A>\frac{1}{\lambda \log \omega^{-1}}$. This shows that the condition $A>\frac{1}{\lambda \log \omega^{-1}}$ implies ALG-strong polynomial tractability with $\varepsilon$-exponent $2\lambda$. This means that the optimal $\varepsilon$-exponent of ALG-strong polynomial tractability is at most $$ 2 \inf\left\{\lambda >0 \ : \ A > \frac{1}{\lambda \log \omega^{-1}}\right\}= \frac{2}{A \log \omega^{-1}},$$ which is a slight improvement of the corresponding result in \cite[Theorem~4.2]{KPW14b}. We remark that \cite[Theorem~4.2]{KPW14b} also tells us that the $\varepsilon$-exponent does not exceed the value 2. 

Furthermore, \eqref{bd:Nstd} with $\lambda=1$ implies that for $\tau,\sigma>0$ we have $$\limsup_{d+\varepsilon^{-1} \rightarrow \infty} \frac{\log N(\varepsilon,d)}{d^\tau + \varepsilon^{-\sigma}} \le 2 A_1 \limsup_{d \rightarrow \infty} \frac{1}{d^\tau} \sum_{j=1}^d \omega^{a_j}.$$ This means that we have ALG-uniform weak tractability (cf. \cite{S13}) if $\sum_{j=1}^d \omega^{a_j}=o(d^\tau)$ for all $\tau>0$, which in turn is equivalent to $\sum_{j=1}^d \omega^{a_j}=O_t(d^t)$ for all $t >0$ (the index $t$ in the $O$-notation indicates that the implied factor depends on $t$). If $\sum_{j=1}^d \omega^{a_j}=O((\log d)^t)$ for some $t>1$, then \eqref{bd:Nstd} implies ALG-quasi-polynomial tractability (cf. \cite{GW11}). If, however, $\sum_{j=1}^d \omega^{a_j}=O(\log d)$, which is certainly satisfied if $\frac{a_j}{\log j} \ge \frac{1}{\log \omega^{-1}}$ for sufficiently large $j$, then we even have ALG-polynomial tractability. The latter is the already known sufficient condition from \cite[Theorem~4.2]{KPW14b}. We see that the sufficient conditions for ALG-uniform weak, ALG-quasi-polynomial and ALG-polynomial tractability are very tight.  For example, assume that we have $\sum_{j=1}^d \omega^{a_j} \le d^t/\sqrt{t}$ for all $t>0$, which implies ALG-uniform weak tractability. Then, with $t=1/n$, $n \in \NN$, we obtain $$\left(\sum_{j=1}^d \omega^{a_j}\right)^n \le n^{n/2} d.$$ Dividing by $n!$ and summing up yields $$\exp\left(\sum_{j=1}^d \omega^{a_j}\right) \le C\, d$$ with $C:=1+\sum_{n=1}^{\infty} n^{n/2}/n! < \infty$ according to the root test. This implies $\sum_{j=1}^d \omega^{a_j} \le \log d + \log C$, and hence we even have ALG-polynomial tractability. 
\end{remark}

\paragraph{Acknowledgment.} The author is grateful to Peter Kritzer and an anonymous referee for valuable comments and remarks.

\noindent\textbf{Author's address:}\\
\noindent Institut f\"ur Finanzmathematik und Angewandte Zahlentheorie, Johannes Kepler Universit\"at Linz, Altenbergerstra{\ss}e 69, 4040 Linz, Austria,
email: friedrich.pillichshammer(AT)jku.at

\end{document}